\documentclass[preprint,review,10pt]{amsart}
\usepackage{amsmath,amssymb,amsfonts,xspace}
\newtheorem{theorem}{Theorem}[section]

\theoremstyle{definition}
\newtheorem{definition}[theorem]{Definition}
\newtheorem{example}[theorem]{Example}

\newtheorem{corollary}[theorem]{Corollary}
\newtheorem{lem}[theorem]{Lemma}

\theoremstyle{remark}

\numberwithin{equation}{section}

\begin{document}

\title{On 1-absorbing prime hyperideal and some  of its generalizations  }

\author{M. Anbarloei}
\address{Department of Mathematics, Faculty of Sciences,
Imam Khomeini International University, Qazvin, Iran.
}

\email{m.anbarloei@sci.ikiu.ac.ir }





\begin{abstract}
Let $R$ be a multiplicative hyperring. In this paper,   we define  the concept of 1-absorbing prime hyperideals which is a generalization of the prime hyperideals.  A proper hyperideal  $I$ of $R$ is called a 1-absorbing prime hyperideal  if for nonunit elements $x,y,z \in R$, $x \circ y \circ z \subseteq I$ , then either $x\circ y \subseteq I$ or $z \in  I$. Several  properties of the hyperideals are provided. Moreover, we introduce  the notions of 1-absorbing primary hyperideals, strongly 1-absorbing primary hyperideals and weakly 1-absorbing primary hyperideals which are generalizations of the 1-absorbing prime hyperideals and then we will show some properties of them.
\end{abstract}
\maketitle
\section{Introduction}
The hyperstructure theory was introduced by Marty in 1934, at the 8th Congress of Scandinavian Mathematicians
\cite{sorc1}, when he defined the hypergroups and began to investigate their properties with applications to groups,  algebraic functions and rational fractions. Later on,
many researchers have worked on this new field of modern algebra and developed it
 \cite {sorc2, sorc3, sorc4, hq6, hq7, davvaz1, hq8, hq9, hq10}. In a classical algebraic structure, the composition of two elements is an element, but in an algebraic hyperstructure, the composition of two elements is a set. 
A well-known type
of a hyperring is called the Krasner hyperring \cite{hq12}. The hyperring  is a hypercompositional
structure $(S, +, \cdot)$ where $(S, +)$ is a canonical hypergroup and $(S, \cdot)$ is a semigroup in which the zero element is absorbing and the operation $\cdot$ is a two-sided distributive one over the hypercomposition $+$.  
In 1982, R. Rota initiated the study of multiplicative hyperring \cite{rota} which was subsequently investigated by many authors \cite{hq1, hq2, hq3, hq4, hq5}. In the hyperring, the multiplication is a hyperoperation, while the addition is an operation. The hyperring in which the additions and the multiplications are hyperoperations was introduced by De Salvo in \cite{hq11}.

In the theory of rings, the key role of the notion of prime ideal as a generalization of the notion of prime number in the ring $\mathbb {Z}$ is  undeniable. The notion of primeness of hyperideal in a multiplicative hyperring was conceptualized by Procesi and Rota in \cite{hq4}.
The notions of prime and primary hyperideals in multiplicative hyperrings were fully studied by Dasgupta in \cite{ref1}. Badawi \cite{hq17} introduced and studied a generalization of prime ideals called
2-absorbing ideals and this notion is further generalized by Anderson and Badawi \cite{hq18, hq19}. 
In \cite{hq13}, Ghiasvand introduced the concept of 2-absorbing
hyperideal in a multiplicative hyperring which is a generalisation of prime hyperideals. Several authors have extended and generalized this concept in several ways \cite{mah, hq14, hq15, hq16}.  The concept of 1-absorbing prime ideals which is another  generalization of prime ideals was introduced in \cite{hq20}.

Motivated from the concept of the prime hyperideals, in this paper, we introduce and study the concept of 1-absorbing prime hyperideals and some  of its generalizations in a multiplicative hyperring.  Several properties of them are provided.   \\
 The paper is orgnized as follows. In Section 2, we have given some basic definitions and results of multiplicative hyperrings which we need to develop our paper. 
In Section 3, we  introduce the concept of 1-absorbing prime hyperideals and give some basic properties of them. For example we show    (Theorem \ref{o1}) that if  $I$ is a 1-absorbing prime hyperideal of $R$, then $\sqrt{I}$ is a prime hyperideal of $R$. Furthermore, $(I:z)$ is a prime hyperideal of $R$ for every nonunit element $z \in R-I$.  In Section 4, we study  a generalization of the  1-absorbing prime hyperideals which is called  1-absorbing primary hyperideals. In particular, we discuss the relations between 1-absorbing primay hyperideals and primary hyperideals. In Section 5, the notion of strongly 1-absorbing primay hyperideals is studied. For example, it is shown (Theorem \ref{54}) that there exists a strongly 1-absorbing primary hyperideal of $R$ if and only if $\sqrt{0}$ is a prime hyperideal or $R$ is a local multiplicative hyperring.   In the final section, we investigate the concept of weakly 1-absorbing primary hyperideals of $R$.

\section{Preiliminaries}
Recall first the basic terms and definitions from the hyperring theory. 
A {\it multiplicative hyperring} is an abelian group $(R,+)$ in which a hyperoperation $\circ $ is defined satisfying the following:

\begin{itemize}
\item[\rm(i)]~ for all $a, b, c \in R$, we have $a \circ (b \circ c)=(a \circ b) \circ c$;
\item[\rm(ii)]~for all $a, b, c \in R$, we have $a\circ (b+c) \subseteq a\circ b+a\circ c$ and $(b+c)\circ a \subseteq b\circ a+c\circ a$;
\item[\rm(iii)]~for all $a, b \in R$, we have $a\circ (-b) = (-a)\circ b = -(a\circ b)$.
\end{itemize}
If in (ii) the equality holds then we say that the multiplicative hyperring is strongly distributive.

Let $A$ and $B$ be two nonempty subsets  of $R$ and $r \in R$. Then  we define
\[ A \circ B=\bigcup_{x \in A,\ y \in B}x \circ y, \ \ \ \ A \circ r=A \circ \{r\}\]

A non empty subset $I$ of a multiplicative hyperring $R$ is a {\it hyperideal} if
\begin{itemize}
\item[\rm(i)]~ If $a, b \in I$, then $a - b \in I$;

\item[\rm(iii)]~ If $x \in I $ and $r \in R$, then $rox \subseteq I$.
\end{itemize}

An element $e \in R$ is said to be {\it scalar identity} if $a=aoe$ for all  $a \in R$.

\begin{definition} \cite{davvaz1}
Let $(R_1, +_1, \circ _1)$ and $(R_2, +_2, \circ_2)$ be two multiplicative hyperrings. A mapping from
$R_1$ into $R_2$ is said to be a {\it good homomorphism} if for all $x,y \in R_1$, $\phi(x +_1 y) =\phi(x)+_2 \phi(y)$ and $\phi(x\circ_1y) = \phi(x)\circ_2 \phi(y)$.
\end{definition}
\begin{definition} \cite{ref1}
A  proper hyperideal $P$ of $R$ is called a {\it prime hyperideal} if $x\circ y \subseteq P$ for $x,y \in R$ implies that $x \in P$ or $y \in P$. The intersection of all prime hyperideals of $R$ containing $I$ is called the prime radical of $I$, being denoted by $\sqrt{I}$. If the multiplicative hyperring $R$ does not have any prime hyperideal containing $I$, we define $\sqrt{I}=R$. 
\end{definition}
\begin{definition} \cite{amer2}
A proper hyperideal $I$ of $R$ is {\it maximal} in R if for
any hyperideal $J$ of $R$ with  $I \subseteq  J \subseteq  R$ then $J = R$. Also, we say that $R$ is a local multiplicative hyperring, if it has just one maximal hyperideal.
\end{definition}
Let {\bf C} be the class of all finite products of elements of $R$ i.e. ${\bf C} = \{r_1 \circ r_2 \circ  . . .  \circ r_n \ : \ r_i \in R, n \in \mathbb{N}\} \subseteq P^{\ast }(R)$. A hyperideal $I$ of $R$ is said to be a {\bf C}-hyperideal of $R$ if, for any $A \in {\bf C}, A \cap I \neq \varnothing $ implies  $A \subseteq I$.
Let I be a hyperideal of $R$. Then, $D \subseteq \sqrt{I}$ where $D = \{r \in R: r^n \subseteq I \ for \ some \ n \in \mathbb{N}\}$. The equality holds when $I$ is a {\bf C}-hyperideal of $R$(\cite {ref1}, proposition 3.2). In this paper, we assume that all hyperideals are {\bf C}-hyperideal.
\begin{definition} \cite{ref1}
A nonzero proper hyperideal $Q$ of $R$ is called a {\it primary hyperideal} if $x\circ y \subseteq Q$ for $x,y \in R$ implies that $x \in Q$ or $y \in \sqrt{Q}$. Since $\sqrt{Q}=P$ is a prime hyperideal of $R$ by Propodition 3.6 in \cite{ref1}, $Q$ is referred to as a P-primary hyperideal of $R$.
\end{definition}
\begin{definition} \cite{amer2}
Let $R$ be a multiplicative hyperring. Then we call $M_n(R)$ as
the set of all hypermatixes of $R$. Also, for all $A = (A_{ij})_{n \times n}, B = (B_{ij})_{n \times n} \in P^\star (M_n(R)), A \subseteq B$ if and only if $A_{ij} \subseteq B_{ij}$. 
\end{definition}
\begin{definition} 
\cite{amer2} Let $R$ be commutative multiplicative hyperring and $e$ be an identity (i. e., for all $a \in R$, $a \in a\circ e$). An element $x \in R$ in  is called {\it unit}, if there exists $y \in R$, such that $e \in x\circ y$. Denote the set of all unit elements in $R$ by
$U(R)$.
\end{definition}
\begin{definition} \cite{hq21}
A hyperring $R$ is called an {\it integral hyperdomain}, if for all $x, y \in  R$,
$0 \in  x \circ y$ implies that $x = 0$ or $y = 0$.
\end{definition}
\begin{definition} 
A hyperring $R$ is said to be a {\it reduced hyperring} if it has no
nilpotent elements. That is, if $x^n = 0$ for $x \in R$ and a natural number $n$, then $x = 0$.
\end{definition}
\begin{definition}
A multiplicative hyperring $F$ is called a {\it hyperfield} if  every non-zero element of $F$ is unit.
\end{definition}
\begin{definition} \cite{amer3}
 Let $R$ be a multiplicative hyperring. An  element $r \in R$  is called {\it regular} if there exists $x \in  R$
such that $r \in r^2\circ x$. So,  $R$ is called {\it regular multiplicative hyperring}, if all of elements in $R$ are regular elements. The set of all regular elements in $R$ is denoted by $V(R)$. 
\end{definition}
\begin{definition}
Let $I,J$ be two hyperideals of $R$ and $x \in R$. Then define:
\[(I:a)=\{r \in R : r\circ a \subseteq I \}\]
\end{definition}
\section{1-absorbing prime hyperideals }
 \begin{definition} 
Let $I$ be a proper hyperideal of $R$.   $I$ is called a 1-absorbing prime hyperideal  if for nonunit elements $x,y,z \in R$, $x \circ y \circ z \subseteq I$ , then either $x\circ y \subseteq I$ or $z \in  I$. 
\end{definition}
It is clear that every prime hyperideal of $R$ is a 1-absorbing prime hyperideal of $R$. Every 1-absorbing prime hyperideal of $R$ is a 2-absorbing hyperideal of $R$.

\begin{example}
In the multiplicative hyperring of integers $\mathbb{Z}_A$ with $A=\{2,3\}$ such that for any $x,y \in \mathbb{Z}_A,  x \circ y =\{x.a.y \ \vert \ a \in A\}$, every principal hyperideal generated by prime integer is a 1-absorbing prime 
hyperideal. 
\end{example}

\begin{theorem} \label {o1}
If $I$ is a 1-absorbing prime hyperideal of $R$, then $\sqrt{I}$ is a prime hyperideal of $R$. Furthermore, $(I:z)$ is a prime hyperideal of $R$ for every nonunit element $z \in R-I$.
\end{theorem}
\begin{proof}
Suppose that  $I$ is a 1-absorbing prime hyperideal of $R$ and $a \circ b \subseteq \sqrt{I}$ for some nonunit $a, b \in R$. Therefore $(a \circ b)^s \subseteq I$ for a positive integer $s$. Clearly, we have $a ^t a^{s-t}b^s \subseteq I$ for some positive integer $t <s$. Since the hyperideal $I$ of $R$ is 1-absorbing prime, then we get either $a^s \subseteq I$ or $b^s \subseteq I$. This means either $a \in \sqrt{I}$ or $b \in \sqrt{I}$. Consequently, the hyperideal $\sqrt{I}$ of $R$ is prime. Now, let $x \circ y \subseteq (I:z)$ for some nonunit elements $x,y \in R$. Assume that  $x \notin (I:z)$. This means $x \circ z \nsubseteq  I$. Since $x \circ z \circ y \subseteq I$ and the hyperideal $I$ is 1-absorbing prime, then we get $y \in I \subseteq (I:z)$. Consequently, the hyperideal $(I:z)$ is prime.
\end{proof}
The following lemma is needed in the proof of our next result.
\begin{lem} \label{23}
Let for every nonunit element $u$ of $R$ and for every unit $v$ of $R$, $u+v$ be a unit element of $R$.  Then $R$ is a
local multiplicative hyperring.
\end{lem}
\begin{proof}
Let $M_1$ and $M_2$ be two maximal hyperideals of $R$. Hence for some $m_1 \in M_1$ and $m_2 \in M_2$ we have $m_1+m_2=1$. By Theorem  3.4 in \cite{amer2}, $1-m_1=m_2$ is a unit element of $R$ which means  $M_2$ contains a unit element. This is a contradiction. Thus $R$ is a local multiplicative hyperring.
\end{proof}
\begin{theorem}
If a hyperideal $I$ of $R$ is 1-absorbing prime that  is not prime, then $R$ is a local multiplicative hyperring.
\end{theorem}
\begin{proof}
Let $I$ is a 1-absorbing prime hyperideal of $R$. Assume that $I$ is not prime hyperideal. Suppose that  we have $x \circ y \subseteq I$ for some $x,y \in R$ but $x,y \notin I$. Assume that $a$ is a nonunit element and $b$ is a unit element of $R$. Let $a+b$ is a nonunit element of $R$. Since $a \circ x \circ y \subseteq I$ and the hyperideal $I$ of $R$ is 1-absorbing prime, then we get $a \circ x \subseteq I$. Since $(a+b)\circ x \circ y \subseteq I$ and $y \notin I$, then we have $(a+b) \circ x \subseteq I$ which implies $b \circ x \subseteq I$. This means $x \in I$ which is a contradiction. Thus $a+b$ is a unit element of $R$. Hence we conclude that $R$ is a local multiplicative hyperring, by Lemma \ref{23}.
\end{proof}
\begin{theorem} \label{AA1}
Let $I$ be a 1-absorbing prime hyperideal of $R$. If $x \circ y \circ J \subseteq I$ for nonunit elements $x,y \in R$ and all proper hyperideals $J$ of $R$. Then $x \circ y \subseteq I$ or $J \subseteq I$.
\end{theorem}
\begin{proof}
Let  $x \circ y \circ J \subseteq I$ for nonunit elements $x,y \in R$ and all proper hyperideals $J$ of $R$ such that $x \circ y \nsubseteq I$ and $J \nsubseteq I$. Therefore there exists an element $a \in J$ but $a \notin I$. Then we have $x \circ y \circ a \subseteq I$ such that  $x \circ y \nsubseteq I$ and $a \notin I$. This is a contradiction, since $I$ is 1-absorbing hyperideal of $R$.
\end{proof}
\begin{theorem} \label{AAA1}
Let $I$ be a $P$-primary  hyperideal of $R$ such that for every $c \in P-I$, $(P^2:c) \subseteq I$. Then hyperideal $I$ of $R$  is 1-absorbing prime.
\end{theorem}
\begin{proof}
Let $I$ be a $P$-primary  hyperideal of $R$ such that for every $a \in P-I$, $(P^2:a) \subseteq I$. Let for $a,b,c \in R$, $a \circ b \circ c \subseteq I$ but neither $a \circ b \nsubseteq I$ nor $c \notin I$. Therefore we have $c \in P-I$, because hyperideal $I$ of $R$  is $P$-primary hyperideal. Thus we get $a \circ b \subseteq (P^2:c) \subseteq I$, a contradiction. Then we have either $a \circ b \subseteq I$ or $c \in I$ which means hyperideal $I$ or $R$ is  1-absorbing prime.  
\end{proof}

\begin{theorem} \label{11126} 
Let $R$ be a multiplicative hyperring with scalar
identity 1 and $I$ be a hyperideal of $R$. If $M_n(I)$ is a 1-absorbing hyperideal of
$M_n(R)$, then $I$ is a 1-absorbing hyperideal of $R$. 
\end{theorem}
\begin{proof}
Suppose that for $x,y,z \in R$ , $x\circ y \circ z \subseteq I$. Then 
\[ \begin{pmatrix}
x\circ y \circ z & 0 & \cdots & 0\\
0 & 0 & \cdots & 0\\
\vdots& \vdots & \ddots \vdots\\
0 & 0 & \cdots & 0 
\end{pmatrix}
\subseteq M_n(I) \]
It is clear that 
\[ \begin{pmatrix}
x \circ y \circ z  & 0 & \cdots & 0\\
0 & 0 & \cdots & 0\\
\vdots& \vdots & \ddots \vdots\\
0 & 0 & \cdots & 0
\end{pmatrix}
=
\begin{pmatrix}
x \circ y  & 0 & \cdots & 0\\
0 & 0 & \cdots & 0\\
\vdots& \vdots & \ddots \vdots\\
0 & 0 & \cdots & 0
\end{pmatrix}
\begin{pmatrix}
z & 0 & \cdots & 0\\
0 & 0 & \cdots & 0\\
\vdots& \vdots & \ddots \vdots\\
0 & 0 & \cdots & 0
\end{pmatrix}
.\]
Since $M_n(I)$ is a 1-absorbing hyperideal of $M_n(R)$, then we have 
\[ \begin{pmatrix}
x \circ y  & 0 & \cdots & 0\\
0 & 0 & \cdots & 0\\
\vdots& \vdots & \ddots \vdots\\
0 & 0 & \cdots & 0
\end{pmatrix} \subseteq M_n(I)\]
or

\[\begin{pmatrix}
z & 0 & \cdots & 0\\
0 & 0 & \cdots & 0\\
\vdots& \vdots & \ddots \vdots\\
0 & 0 & \cdots & 0
\end{pmatrix}
\in M_n(I).\]
It follows that either $x \circ y \subseteq I$ or $ z \in I$. Therefore $I$ is a 1-absorbing hyperideal of $R$.
\end{proof}
Let $(R, +, o)$ be a hyperring. We define the relation $\gamma$ as follows:\\ 
$a \gamma b$ if and only if $\{a,b\} \subseteq U$ where $U$ is a finite sum of finite products of
elements of R, i.e.,\\
$a \gamma b \Longleftrightarrow \exists z_1, ... , z_n \in R$ such that $\{a, b\} \subseteq \sum_{j \in J} \prod_{i \in I_j} z_i; \ \ I_j, J \subseteq \{1,... , n\}$.

We denote the transitive closure of $\gamma$ by $\gamma ^{\ast}$. The relation $\gamma ^{\ast}$ is the smallest equivalence relation on a multiplicative hyperring $(R, +, o)$ such that the
quotient $R/\gamma ^{\ast}$, the set of all equivalence classes, is a fundamental ring. Let $\mathfrak{U}$
be the set of all finite sums of products of elements of R we can rewrite the
definition of $\gamma ^{\ast}$ on $R$ as follows:\\
$a\gamma ^{\ast}b \Longleftrightarrow \exists z_1, ... , z_n \in R$ with $z_1 = a, z_{n+1 }= b$ and $u_1, ... , u_n \in \mathfrak{U}$ such that
$\{z_i, z_{i+1}\} \subseteq u_i$ for $i \in \{1, ... , n\}$.

Suppose that $\gamma ^{\ast}(a)$ is the equivalence class containing $a \in R$. Then, both
the sum $\oplus$ and the product $\odot$ in $R/\gamma^{\ast}$ are defined as follows:$\gamma ^{\ast}(a) \oplus \gamma^{\ast}(b)=\gamma^{\ast}(c)$ for all $c \in  \gamma^{\ast}(a) + \gamma ^{\ast}(b)$ and $\gamma^{\ast}(a) \odot \gamma^{\ast}(b)=\gamma^{\ast}(d)$ for all $d \in  \gamma^{\ast}(a) o \gamma^{\ast}(b)$
Then $R/\gamma^{\ast}$ is a ring, which is called a fundamental ring of $R$ (see also \cite{sorc4}).

\begin{theorem}
Let $R$ be a multiplicative hyperring with scalar identity $e$. Then the hyperideal $I$ of $R$ is 1-absorbing hyperideal if and only if $I/\gamma ^{\ast}$ is a 1-absorbing ideal of $R/\gamma ^{\ast}$. 
\end{theorem}
\begin{proof}
($\Longrightarrow$) Let for $x, y,z \in R/\gamma ^{\ast}, \ x \odot y \odot z \in I/\gamma ^{\ast}$ . Then, there exist $a, b,c\in R$ such
that $x =\gamma^{\ast}(a), y = \gamma^{\ast}(b) , z= \gamma^{\ast}(c)$ and $x \odot y \odot z= \gamma^{\ast}(a) \odot \gamma^{\ast}(b) \odot \gamma^{\ast}(c)=\gamma^{\ast}(a \circ  b \circ c)$. Since $\gamma^{\ast} (a) \odot \gamma^{\ast}(b) \odot \gamma^{\ast}(c)=
\gamma^{\ast}(a\circ b \circ c) \in I/\gamma^{\ast}$, then $a \circ b \circ c\subseteq I$. Since $I$ is a 1-absorbing hyperideal of $R$, then $a \circ b \subseteq I$ or $c \in I$. Hencey $x \odot y= \gamma^{\ast}(a) \odot  \gamma^{\ast}(b)=\gamma^{\ast}(a \circ b)\in I/\gamma ^{\ast}$ or $z=\gamma^{\ast}(c)\in I/\gamma ^{\ast}$.
Thus $I/\gamma ^{\ast}$ is a 1-absorbing ideal of $R/\gamma ^{\ast}$.\\
($\Longleftarrow$) Suppose that $a \circ b \circ c \subseteq I$ for $a, b, c \in R$. Then $\gamma^{\ast}(a), \gamma^{\ast}(b), \gamma^{\ast}(c)\in R/\gamma^{\ast} $ and
$\gamma^{\ast}(a) \odot\gamma^{\ast}(b) \odot\gamma^{\ast}(c)= \gamma^{\ast}(a \circ  b \circ c) \in I/\gamma^{\ast}$. Since $I/\gamma ^{\ast}$ is a 1-absorbing ideal of $R/\gamma ^{\ast}$, then we have $\gamma^{\ast}(b) \odot \gamma^{\ast}(b)= \gamma^{\ast}(a \circ b)\in I/\gamma^{\ast}$ or  $\gamma^{\ast}(c) \in I/\gamma^{\ast}$. It means that $a \circ b \subseteq I$ or $c \in I$. Hence $I$ is a 1-absorbing hyperideal of $R$.
\end{proof}
\begin{theorem}
Let $I$ be a proper hyperideal of $R$. Then the following statements are equivalent.
\begin{itemize}
\item[{(1)}] $I$ is a 1-absorbing prime hyperideal of $R$.
\item[{(2)}] If $J \circ H \circ K \subseteq I$ for some proper hyperideals $J, H,K$ of $R$, then either $J \circ H \subseteq I$ or $K \subseteq I$.
\end{itemize}
\end{theorem}
\begin{proof}
$(1) \Longrightarrow (2)$ Let $I$ be  a 1-absorbing prime hyperideal of $R$ and $J \circ H \circ K \subseteq I$ for some proper hyperideals $J, H,K$ of $R$ such that $J \circ H \nsubseteq I$. Thus there exist nonunit elements $j \in J$ and $h \in H$ such that $j \circ h \nsubseteq I$. Since $j \circ h \circ K \subseteq I$ and $j \circ h \nsubseteq I$, then we get $K \subseteq I$ by Lemma \ref{AA1}.
\\$(2) \Longrightarrow (1)$ Let $x \circ y \circ z \subseteq I$ for some nonunit elements $x,y,z \in R$ such that $x \circ y \nsubseteq I$. Assume that $J=\langle x \rangle$, $H=\langle y \rangle$ and $K=\langle z \rangle$. This means $J\circ H \circ K \subseteq I$. Since $J \circ H \nsubseteq I$, then we conclude that $K \subseteq I$ and so $z \in I$.
\end{proof}
\section{1-absorbing primary hyperideals }
\begin{definition} 
Let $I$ be a proper hyperideal of $R$.   $I$ is called a 1-absorbing primary hyperideal  if for nonunit elements $x,y,z \in R$, $x \circ y \circ z \subseteq I$ , then either $xoy \subseteq I$ or $z \in  \sqrt{I}$. 
\end{definition}

   Recall that a hyperideal $I$  of $R$ is said to be a   2-absorbing primary  hyperideal if for $x,y,z \in R$, $x \circ y \circ z \subseteq I$, then $xoy \subseteq I$ or $x \circ z \subseteq \sqrt{I}$  or $y \circ z \subseteq \sqrt{I}$.
\begin{theorem} \label{21}
\begin{itemize}
\item[\rm{(1)}]~ Every primary hyperideal of $R$ is a 1-absorbinh primary hyperideal.
\item[\rm{(2)}]~ Every 1-absorbinh primary hyperideal of $R$ is a 2-absorbing primary hyperideal.
\end{itemize}
\end{theorem}
\begin{example}
In the multiplicative hyperring of integers $\mathbb{Z}_A$ with $A=\{2,3\}$ such that for any $x,y \in \mathbb{Z}_A,  x \circ y =\{x.a.y \ \vert \ a \in A\}$, hyperideal $<5>=\{5n \ \vert \ n \in \mathbb{Z}\}$  is a 1-absorbinh primary hyperideal. 
\end{example}
\begin{example}
Let $(\mathbb{Z},+,.)$ be the ring of integers. We define the hyperoperation $a \circ b =\{2x.y,3xy\}$ for all $a,b \in \mathbb{Z}$. Then $(\mathbb{Z},+,\circ)$ is a multiplicative hyperring. The hyperideals $<2>$ and $<3>$ of $\mathbb{Z}$ are 1-absorbinh primary hyperideals.
\end{example}
\begin{theorem} \label{22}
Let $I$ be a hyperideal of $R$. If $I$ is 1-absorbing primary  then
$\sqrt{I}$ is a prime hyperideal of $R$.
\end{theorem}
\begin{proof}
Let for some $x,y \in R$, $x \circ y \subseteq \sqrt{I}$. We suppose that $x,y$ are nonunit elements of $R$. We assume that for $n \geq 2$, $(x \circ y)^n \subseteq I$. Since $(x \circ y)^n=x^n \circ y^n=x^{n-1} \circ x \circ y^n \subseteq I$, we have $x^{n-1} \circ x=x^n \subseteq I$ or $y ^n \subseteq \sqrt{I}$, because $I$ is a 1-absorbing primary hyperideal of $R$. Thus $x \in \sqrt{I}$ or $y \in \sqrt{I}$ which means $\sqrt{I}$ is a prime hyperideal of $R$.
\end{proof}

\begin{theorem} \label{24}
Let $I$ be   a 1-absorbing primary hyperideal of $R$ such that is not a primary hyperideal. Then $R$ is a local multiplicative hyperring. 
\end{theorem}
\begin{proof}
Let hyperideal $I$ of $R$ be 1-absorbing primary such that is not primay. Let  for some nonunit elements $x,y \in R$, $x \circ y \subseteq I$. By the assumption,  we conclude that $x \notin I$ and $y \notin \sqrt{I}$. Let $u$ be a nonunit element of $R$. Then we get $u \circ x \circ y \subseteq I$. Since hyperideal $I$ is  1-absorbing
primary  and $y \notin \sqrt{I}$, then $u \circ x \subseteq I$. Let $v$ be a unit element of $R$ such that $u+v$ is a nonunit element of $R$. It is clear that $(u+v) \circ x \circ y \subseteq I$. Since $I$ is a 1-absorbing primary hyperideal of $R$, we get  $(u+v)\circ x =u \circ x+v \circ x \subseteq I$, because  $y \notin \sqrt{I}$. Since $u \circ x \subseteq I$, we have $v \circ x \subseteq I$ which implies $x \in I$ which is a contradiction. Hence $u+v$ is a unit
element of $R$. By Lemma \ref{23}, $R$ is a local multiplicative hyperring.
\end{proof}
In the following Theorem, we show that if $R$ is not a local multiplicative hyperring, then every 1-absorbing
primary hyperideal is primary.
\begin{theorem} \label{25}
Suppose that  $R$ is not a local multiplicative hyperring. Then 
$I$  is a 1-absorbing primary hyperideal of $R$ if and only if $I$ is a primary hyperideal of $R$.
\end{theorem}
\begin{proof}
$\Longrightarrow$  Let $I$ be a 1-absorbing primary hyperideal of $R$ and let $x \circ y \subseteq I$ for some nonunit elements $x, y \in R$ such that  $y \notin \sqrt{I}$. By Lemma \ref{23}, there exist a nonunit element $u \in R$ and a unit element $v \in R$ such that $u+v$ is a nonunit element of $R$, because $R$ is not a local multiplicative hyperring. It is clear that $u \circ x \circ y \subseteq I$. Since $I$  is a 1-absorbing primary hyperideal of $R$ and $y \notin \sqrt{I}$, then we get $u \circ x \subseteq I$. Moreover, $(u+v) \circ x \circ y \subseteq I$. Then $(u+v)\circ x=u \circ x+v \circ x \subseteq I$, because $I$ is a 1-absorbing primary hyperideal of $R$ and $y \notin \sqrt{I}$. Thus $v \circ x \subseteq I$ which implies $x \in I$. Consequently $I$ is a primary hyperideal of $R$.\\
$\Longleftarrow$ It is clear.
\end{proof}
\begin{theorem} \label{26}
Let  $R_1$ and $R_2$ be two multiplicative hyperrings. Let $I$ be a hyperideal of $R_1 \times R_2$. The following statements are equivalent.
\begin{itemize}
\item[{(1)}]~ $I$ is a 1-absorbing primary hyperideal of $R_1 \times R_2$.
\item[{(2)}]~ $I$ is a primary hyperideal of $R_1 \times R_2$.
\item[{(3)}]~ $I=I_1 \times R_2$ for some primary hyperideal $I_1$ of $R_1$.
\end{itemize}
\end{theorem}
\begin{definition}
A nonzero nonunit element $x$ of $R$ is called irreducible if $x \in 
x_1 \circ x_2$ for some $x_1, x_2 \in R$, then $x_1$ is a unit of $R$ or $x_2$ is a unit of $R$.  .
\end{definition}
\begin{definition}
A nonzero element $x$ of $R$ is called prime if $x_1 \circ x_2 \subseteq x \circ r$ for some $x_1,x_2,r \in R$, then  $x_1 \in x \circ r_1$ for some $r_1 \in R $ or   $x_2 \in x \circ r_2$ for some $r_2 \in R_2$.
\end{definition}
\begin{lem} \label{27}
Let $R$ be a local strongly distributive multiplicative hyperring. If $x$ is a nonzero prime element of $R$, then
$x$ is an irreducible element of R.
\end{lem}
\begin{proof}
Let $x$ is a nonzero prime element of $R$ such that for some $x_1, x_2 \in R$, $x \in x_1 \circ x_2$.  We may assume that for some $r_ 1\in R$, $x_1 \in x \circ r_1$. Now, we show that $x_2$ is a unit element of $R$. Then we obtain $x \in x \circ r_1 \circ x_2$. Hence  $0 \in x \circ (1-r_1 \circ x_2) = x-x \circ r_1 \circ x_2$. Let $x_2$ be a
nonunit element of $R$.  Since $R$ is  a local multiplicative hyperring, then $1-r_1 \circ x_2$ is a unit element of $R$ which implies  $x=0$ which is a contradiction.
Thus  $x_2$ is a unit element of $R$. Consequently,  $x$ is an irreducible element of $R$.
\end{proof}
\begin{theorem} \label{28}
Let $R$ be a local strongly distributive multiplicative hyperring with maximal hyperideal $M$. If  element $x$ of $M$ is  a
nonzero prime element of $R$ with  $M \neq  x\circ R$, then $x \circ M$ is a 1-absorbing primary hyperideal of $R$ that is not a primary hyperideal of $R$.
\end{theorem}
\begin{proof}
Let  for some nonunit elements
$a,b,c \in R$, $a \circ b \circ c \subseteq x \circ M$. Let $a \circ b \nsubseteq x \circ M$. Then for all $r_1,r_2 \in M $, $a \notin x \circ r_1$ and $b \notin x \circ r_2$. Since for some $r \in M$, $a \circ b \circ c \subseteq x  \circ r$ and for all $s \in M$, $a \circ b \nsubseteq x \circ s$, then we have $c \in x \circ t$ for some $t \in M$. Consequently, $c \in \sqrt{x \circ M}=x \circ R$. Thus  $x \circ M$ is a 1-absorbing primary hyperideal of $R$.
Now, Since $M \neq  x\circ R$, then there exists $m \in M$ but $m \notin x \circ R$. It is clear that $x \circ m \subseteq x\circ M$. By Lemma \ref{27}, $x$ is an irreducible element of $R$, because $x$ is a nonzero prime element of $R$. Then we conclude that $x \notin x \circ M$. Since $\sqrt{x \circ M}=x \circ R$, then we get $m \notin \sqrt{x \circ M}$ which means $x\circ M$ is not a primary hyperideal of $R$. 
\end{proof}
\begin{theorem} \label{29}
Let $I$ is a 1-absorbing primary hyperideal of $R$. Suppose that $I$ is  not a primary hyperideal of $R$. Then  for some irreducible element $x\in R-I$ and some nonunit
element $y\in R-\sqrt{I}$, $x \circ y \subseteq I$. Moreover, if $x \circ y \subseteq I$ for some nonunit elements $x \in R-I$ and $y \in \sqrt{I}$, then element $x$ of $R$ is irreducible.
\end{theorem}
\begin{proof}
Suppose that $I$ is not a primary hyperideal of $R$. Then for some  nonunit elements $x\in R-I$ and  $y\in R-\sqrt{I}$, $x \circ y \subseteq I$. Assume that $x$ is not irreducible. Hence $x \in a \circ b$ for some nonunit elements $a,b \in R$. Since $I$
is a 1-absorbing primary hyperideal of $R$ and $x \circ y \subseteq a \circ b \circ y \subseteq I$ and $y \notin \sqrt{I}$, then $a \circ b \subseteq I$ which implies $x \in I$ which is a contradiction. Thus $x$ is  irreducible.
\end{proof}
\begin{theorem}  \label{310}
Suppose that  $R$ is  a local hyperring with maximal hyperideal $M$. Let $P$ be a prime hyperideal of $R$ with  $P\subseteq M$. Then hyperideal $P\circ M$ of $R$ is a 1-absorbing primary hyperideal.
\end{theorem}
\begin{proof}
Let $x \circ y \circ z \subseteq P\circ M$ for some nonunit
elements $x,y,z \in R$. Since $\sqrt{P\circ M}=M$, then $x \circ y \circ z \subseteq p.$ Let $x \circ y \nsubseteq P \circ M$. Then $x \notin P$ and $y \notin P$  which means $a \circ b \nsubseteq P$ (note that if $x \in P$ or $y \in P$ then $x \circ y \subseteq P \circ M$). Thus $z \in P=\sqrt{P\circ M}$. Consequently, hyperideal $P \circ M$ of $R$ is a 1-absorbing primary.
\end{proof}
\begin{theorem}  \label{311}
Let hyperideal $I$ be a 1-absorbing primary hyperideal  of $R$ and let $a \in R-I$ be a nonunit element of $R$. Then $(I : a)$ is a primary  hyperideal of $R$.
\end{theorem}
\begin{proof}
 Let $x \circ y  \subseteq  (I : a)$ for some nonunit elements $x,y \in  R$.  Let $x \notin (I:a)$ . Then $a \circ x \nsubseteq I$. Since $x \circ y  \subseteq   (I : a)$, then $a \circ x \circ y  \subseteq I$. 
Since $I$ is a 1-absorbing primary hyperideal  of $R$ and $a \circ x \nsubseteq I$, then $y \in \sqrt{I} \subseteq \sqrt{(I:a)}$ which means hyperideal $(I:a)$ is  primary. 
\end{proof}
\begin{definition}
 Let $I$ be a 1-absorbing primary hyperideal of $R$. By Theorem \ref{22}, $\sqrt{I}=P$ is a prime hyperideal of $R$. Thus $I$ is called $P$-1-absorbing primary hyperideal of $R$.
\end{definition}
\begin{theorem}  \label{41}
Let  $I_1,...,I_n$ be $P$-1-absorbing primary hyperideals of $R$. Then $\bigcap_{i=1}^nI_i$ is a $P$-1-absorbing primary hyperideal of $R$.
\end{theorem}
\begin{proof}
Put $I=\bigcap_{i=1}^nI_i$. It is clear that $\sqrt{I}=P$. Let for some nonunit elements $x,y,z \in R$, $x \circ y  \circ z \subseteq I$ such that $x \circ y \nsubseteq I$. Then there exist $1 \leq i \leq n$ such that $x \circ y \nsubseteq I_i$. Since $I_i$ is a $P$-1-absorbing primary hyperideal of $R$, then $z \in \sqrt{I_i}=P$ which means $z \in \sqrt{I}$. Consequently, $I=\bigcap_{i=1}^nI_i$ is a $P$-1-absorbing primary hyperideal of $R$.
\end{proof}
\begin{theorem} \label{42} 
Let $R_1$ and $R_2$ be multiplicative hyperrings and $ \phi:R_1 \longrightarrow R_2$ be a good homomorphism  such that if $R_2$ is a local multiplicative  hyperring, then $\phi(x)$ is a nonunit of $R_2$ for
every nonunit $x \in R_1$. Then the following statements hold : 
\begin{itemize}
\item[{(1)}]~ If $I_2$ is a 1- absorbing primary hyperideal of $R_2$, then $\phi^{-1} (I_2)$ is a 1- absorbing primary hyperideal of $R_1$.
\item[{(2)}]~ If $\phi $ is an epimorphism and $I_1$ is a 1- absorbing primary hyperideal of $R_1$ containing $Ker(\phi)$, then $\phi(I_1)$ is a 1- absorbing primary hyperideal of $R_2$.
\end{itemize}
\end{theorem}
\begin{proof} 
(1) Let $R_2$ be  a local multiplicative  hyperring. Let for some nonunit elements  $x,y,z \in R_1$, $x \circ  y \circ z \subseteq \phi^{-1}(I_2)$. Thus we have $\phi(x \circ y \circ z)=\phi(x) \circ \phi(y) \circ \phi(z) \subseteq I_2$.  Since $I_2$ is a 1- absorbing primary hyperideal of $R_2$, then we conclude that $\phi(x \circ y)=\phi(x) \circ \phi(y) \subseteq I_2$ which implies $x \circ y \subseteq \phi^{-1}(I_2)$ or $\phi(z) \in \sqrt{I_2}$ which implies $z \in \sqrt{\phi^{-1}(I_2)}=\phi^{-1}(\sqrt{I_2}) $. This means  $\phi^{-1}(I_2)$ is a 1-absorbing primary hyperideal  of $R_1$. Now let $I_2$ be  a 1- absorbing primary hyperideal of $R_2$ such that $R_2$ is not   a local multiplicative  hyperring. By Theorem \ref{24}, hyperideal $I_2$ of $R_2$ is  primary.  Then we conclude that $\phi^{-1}(I_2)$ is a primary hyperideal of $R_1$ which means $\phi^{-1}(I_2)$ is a 1-absorbing primary hyperideal of $R_1$.\\
(2) Let for some nonunit elements $x_2,y_2,z_2 \in R_2$,  $x_2 \circ y_2 \circ z_2 \subseteq \phi(I_1)$. Since $\phi$ is an epimorphism, there exist $x_1,y_1,z_1 \in R_1$ such that $\phi(x_1)=x_2, \phi(y_1)=y_2, \phi(z_1)=z_2$ and so $\phi(x_1 \circ y_1 \circ z_1)=x_2 \circ y_2 \circ z_2\subseteq \phi(I_1)$. Now take any $u \in x_1 \circ  y_1 \circ z_1 $. Then we get $\phi(u) \in \phi(x_1 \circ y_1 \circ z_1) \subseteq \phi(I_1)$ and so $\phi(u) = \phi(w)$ for some $w \in I_1$. This implies that $\phi(u-w) = 0 $, that is, $u-w \in Ker(\phi) \subseteq I_1$ and so $u \in I_1$. Since $I_1$ is a ${\bf C}$-hyperideal of $R_1$, then we conclude that $x_1 \circ y_1 \circ z_1\subseteq I_1$. Since $I_1$ is a 1-absorbing primary hyperideal of $R_1$ then $x_1 \circ y_1 \subseteq I_1$ which implies $x_2 \circ y_2 =\phi(x_1 \circ y_1) \subseteq \phi(I_1)$ or $z_1 \in \sqrt{I_1}$ which implies $z_2=\phi(z_1) \in \phi( \sqrt{I_1})= \sqrt{\phi (I_1)}$ .Consequently, $\phi(I_1)$ is a 1-absorbing primary hyperideal  of $R_2$. 
\end{proof}
\begin{corollary}  \label{43}
Let $I, J$ be two proper hyperideals of $R$ with $J \subseteq I$ such that if $R/J$ is a local multiplicative hyperring, then $x + J$ is a nonunit of $R/J$ for every nonunit $x \in R$.
 Then $I$ is a 1-absorbing primary hyperideal of $R$ if and only if $I/J$ is a 1-absorbing primary hyperideal of $R/I$.
\end{corollary} 
\begin{proof} 
Define $\phi:R \longrightarrow R/J$ by $\phi(r)=r+J$. clearly, $\phi$ is a good epimorphism. Since $Ker (f)=J \subseteq I$ and $I$  is a 1-absorbing primary hyperideal of $R$ , then the claim follows from Theorem \ref{42} (2). Now let $I/J$ be  a 1-absorbing primary hyperideal of $R/J$. Then we conclude that $\phi^{-1}(I/J)=I$ is a 1-absorbing primary hyperideal of $R$ by Theorem \ref{42} (1).
\end{proof}
\begin{theorem} \label{44}
Let the hyperideal $I$ of $R$  be  1-absorbing. If $J$ is a proper hyperideal of $R$ such that $x \circ y \circ J \subseteq I$
for some nonunit elements $x,y \in R$, then $x \circ y \subseteq I$
or $J \subseteq \sqrt{I}$.
\end{theorem}
\begin{proof}
Let $x \circ y \circ J \subseteq I$ such that neither $x \circ y \nsubseteq I$
nor $J \nsubseteq \sqrt{I}$. Thus there exists an element $r \in J$ such that $r \notin \sqrt{I}$. It is clear that  $x \circ y \circ r \subseteq I$ but  $x \circ y   \nsubseteq I$ and $r \notin \sqrt{I}$ which is a contradiction, because  $I$  is a 1-absorbing hyperideal of $R$.
\end{proof}
Next, it is proved that a proper hyperideal $I$ of $R$ is 1-absorbing primary if and only if the inclusion  $J \circ K \circ L \subseteq I$ for any proper hyperideals $J,K,L$ of $R$ implies that either $J \circ K \subseteq I$ or $L \subseteq \sqrt{I}$

\begin{theorem} \label{45}
Let I be a proper hyperideal of $R$. Then 
$I$ is a 1-absorbing primary hyperideal of $R$ if and only if 
for any proper hyperideals $J,K,L$ of $R$ with $J \circ K \circ L \subseteq I$ implies that either $J \circ K \subseteq I$ or $L \subseteq \sqrt{I}$.
\end{theorem}
\begin{proof}
$\Longrightarrow$  Let the hyperideal $I$ of $R$  be  1-absorbing primary. Suppose that  for some proper hyperideals $J,K,L$ of $R$,  $J \circ K \circ L \subseteq I$ such that $J \circ K \nsubseteq I$. This means  there exist nonunit
elements $x \in J$ and $y \in K$ but $x \circ y \nsubseteq I$. Thus we get $x \circ y \circ L \subseteq I$. Since  $I$ is a 1-absorbing primary hyperideal of $R$, then $L \subseteq \sqrt{I}$ by Theorem \ref{44}.\\
$\Longleftarrow$ Let for some nonunit elements $x,y,z \in R$, $x \circ y \circ z \subseteq I$ but $x \circ y \nsubseteq I$. We have $\langle x \rangle \circ \langle y \rangle \circ \langle z \rangle \subseteq I$. Hence $\langle x \rangle \circ \langle y \rangle  \nsubseteq I$. By the assumption we get $\langle z \rangle \subseteq \sqrt{I}$ which means $z \in \sqrt{I}$. Thus the hyperideal $I$ of $R$ is 1-absorbing.
\end{proof}
\section{strongly 1-absorbing primary hyperideals}
\begin{definition} 
Let $I$ be a proper hyperideal of $R$.   $I$ is called a strongly 1-absorbing primary hyperideal  if for nonunit elements $x,y,z \in R$, $x \circ y \circ z \subseteq I$ , then $xoy \subseteq I$ or $z \in  \sqrt{0}$.
\end {definition} 
It is clear that if $I$ is a  strongly 1-absorbing primary hyperideal of $R$ then  $I$ is a 1-absorbing primary
hyperideal.
\begin{example}
Consider the ring $(\{\bar{0},\bar{1},\bar{2},\bar{3}\}=\mathbb{Z},\boxplus,\odot)$ that for each $\bar{x},\bar{y} \in \mathbb{Z}_4$, $\bar{x} \boxplus \bar{y}$ and $\bar{y} \odot \bar{y}$ are the remainder of $\frac{x+y}{4}$ and $\frac{x.y}{4}$, respectively, which $"+"$ and $"."$ are
ordinary addition and multiplication, and $x,y \in \mathbb{Z}$.
\[\begin{tabular}{c|c} 
$\boxplus$ & $\bar{0}$  $\bar{1}$  $\bar{2}$  $\bar{3}$
\\ \hline $\bar{0}$ &  $\bar{0}$  $\bar{1}$  $\bar{2}$ 
 $\bar{3}$ 
\\ $\bar{1}$ & $\bar{1}$  $\bar{2}$  $\bar{3}$  $\bar{0}$ 
\\ $\bar{2}$ & $\bar{2}$ $\bar{3}$  $\bar{0}$  $\bar{1}$ 
\\$\bar{3}$ & $\bar{3}$  $\bar{0}$  $\bar{1}$  $\bar{2}$ 
\end{tabular} \ \ \ \ \ \ \ \ 
\begin{tabular}{c|c} 
$\boxdot$ & $\bar{0}$ \ \ \ \ \ \ \ \ \ \ \   $\bar{1}$ \ \ \  \ \ \ \ \ \ \ \ \ \ \  $\bar{2}$ \ \ \ \ \  \ \ \  \ \  $\bar{3}$
\\ \hline $ \bar{0}$  & $\{\bar{0}\}$\ \ \ \ \ \ \ \ \ \  $\{\bar{0}\}$  \ \ \  \ \ \ \ \  $\{\bar{0}\}$ \ \ \ \ \ \  $\{\bar{0}\}$ 
\\ $\bar{1}$ & $\{\bar{0}\}$ \ \ \ \ \ \ \ \ \  $\mathbb{Z}_4$ \ \ \ \ \ \ $\{\bar{0},\bar{2}\}$ \ \ \ \ \ \ $\mathbb{Z}_4$
\\ $\bar{2}$ & $\{\bar{0}\}$  \ \ \ \ \  $\{\bar{0},\bar{2}\}$ \ \ \ \ \ \ \ $\bar{0}$ \ \ \ \ \ \ \ $\{\bar{0},\bar{2}\}$
\\$\bar{3}$ & $\{\bar{0}\}$ \ \ \ \ \ \ \ $\mathbb{Z}_4$ \ \ \ \ \ \ \ $\{\bar{0},\bar{2}\}$ \ \ \ \ \ \ \ $\mathbb{Z}_4$
\end{tabular}\]
which $(\mathbb{Z}_4,\boxplus,\boxdot)$ is a  multiplicative hyperring.  In the hyperring, hyperideal $\{0,2\}$ is a strongly 1-absorbing primary hyperideal of $\mathbb{Z}_4$.
\end{example} 
\begin{theorem}
If  the hyperideals $I_1$ and $I_2$ of $R$  are strongly 1-absorbing primary then $I_1 \cap I_2$ is a strongly 1-absorbing primary hyperideal of $R$.
\end{theorem}
\begin{theorem} \label{51}
Let $I$ be a proper hyperideal of $R$. Then, $I$ is a strongly 1-absorbing primary hyperideal of $R$ if and only if
$I$ is 1-absorbing primary and $\sqrt{I}=\sqrt{0}$, or
$R$ is a local multiplicative hyperring with maximal hyperideal $M=\sqrt{I}$ such that  $M^2 \subseteq I$.
\end{theorem}
\begin{proof}
$\Longrightarrow$  It is clear that if $I$ is a  strongly 1-absorbing primary hyperideal of $R$ then  $I$ is a 1-absorbing primary
hyperideal. Now, let $\sqrt{I} \neq \sqrt{0}$. Suppose that for some $x, y \in R$, $x \circ y \nsubseteq I$. We get $x \circ y \circ z \subseteq I$ for each $z \in I$. Since hyperideal $I$ is  strongly 1-absorbing primary, then we obtain $z \in \sqrt{0}$ which implies  $I \subseteq \sqrt{0}$ which means $I = \sqrt{0}$ which is a contradiction. Thus for each nonunit elements  $x, y \in R$, $x \circ y \subseteq I$. Let hyperideal $M$ of $R$ be
 maximal. Then $M^2 \subseteq I$ which implies $\sqrt{M^2} \subseteq \sqrt{I}$ which means $M \subseteq \sqrt{I}$. Therefore for each maximal hyperideal $M$ of $R$, 
$M=\sqrt{I}$. Consequently, $R$ is a local multiplicative hyperring.\\
$\Longleftarrow$ Let $I$ be 1-absorbing primary and $\sqrt{I}=\sqrt{0}$. Clearly hyperideal $I$ is trongly 1-absorbing primary. Let $R$ is a local multiplicative hyperring with maximal hyperideal $M=\sqrt{I}$ such that  $M^2 \subseteq I$. Then we have $x \circ y \subseteq M^2 \subseteq I$ for each nonunit elements $x,y \in R$. Then hyperideal $I$ of $R$  is  strongly 1-absorbing primary.
\end{proof}
In view of Theorem \ref{51}, we have the following results.
\begin{corollary}  \label{52}
Let the hyperideal $P$ of $R$  be prime. Then, $P$ is  strongly 1-absorbing primary if and only if $P=\sqrt{0}$ or 
$R$ is a local multiplicative hyperring with maximal hyperideal $P$.
\end{corollary}
\begin{corollary}  \label{53}
Let $R$ is a local multiplicative hyperring with maximal hyperideal $M$ and let $P$ be a prime hyperideal of $R$. Then hyperideal $P \circ M$ of $R$ is  strongly 1-absorbing primary if and only if $P=
\sqrt{0}$ or $P=M$.
\end{corollary}
\begin{proof}
By Theorem 8 in \cite{bad}, the hyperideal $P \circ M$ of $R$ is  1-absorbing primary.   By Theorem \ref{51},  the hyperideal $P \circ M$ of $R$ is  strongly 1-absorbing primary if and only if  $\sqrt{P \circ M}=\sqrt{0}$ or $\sqrt{P \circ M}=M$ and $M^2 \subseteq P \circ M$ which implies $P=\sqrt{0}$ or $P=M$.
\end{proof}
\begin{theorem} \label{54}
Let $R$ be a multiplicative hyperring. Then, there exists a strongly 1-absorbing primary hyperideal of $R$ if and only if $\sqrt{0}$ is a prime hyperideal or $R$ is a local multiplicative hyperring.
\end{theorem}
\begin{proof}
$\Longrightarrow$ Let  $I$ be a strongly 1-absorbing primary hyperideal of $R$. Suppose that $R$ is not a local multiplicative hyperring. Then $\sqrt{I}=\sqrt{0}$ by Theorem \ref{51}. Since $I$ is  a strongly 1-absorbing primary hyperideal of $R$, then $\sqrt{0}$  is a prime hyperideal by Theorem \ref{22}.\\
$\Longleftarrow$ Let $R$ is a local multiplicative hyperring with maximal hyperideal $M$, Then $M$  is a  strongly 1-absorbing primary by Corollary \ref{52}. On the other hand, if $\sqrt{0}$ is a prime hyperideal or $R$, then by Corollary \ref{52}, $\sqrt{0}$ is a  strongly 1-absorbing primary of $R$.
\end{proof}
\begin{corollary} \label{55}
If $R_1$ and $R_2$ are two multiplicative hyperrings. Then, $R=R_1 \times R_2$ has no strongly
1-absorbing primary hyperideal.
\end{corollary}
\begin{proof}
Since $\sqrt{0_R}=\sqrt{0_{R_1}} \times \sqrt{0_{R_2}}$ is
not prime hyperideal in $R$ and $R$ is not a local multiplicative hyperring, then $R$ has no strongly
1-absorbing primary hyperideal by Theorem \ref{54}.
\end{proof}
\begin{theorem} \label{56}
The hyperideal $I$ of $R$ is  strongly 1-absorbing primary if and only if whenever for any proper hyperideal $J,K,L$ of $R$ with $J \circ K \circ L \subseteq I$ implies that either $J \circ K \subseteq I$ or $ L \subseteq  \sqrt{0}$
\end{theorem}
\begin{proof}
$\Longrightarrow$ Let for some proper hyperideal $J,K,L$ of $R$,  $J \circ K \circ L \subseteq I$ but $J \circ K \nsubseteq I$. This means there exist $a \in J$ and $b \in K$ such that $a \circ b \nsubseteq I$. It is clear that for each $c \in L$ , $a \circ b \circ c \subseteq I$. Since $I$ is a  strongly 1-absorbing primary hyperideal of $R$ and $a \circ b \nsubseteq I$, then we get $c \in \sqrt{0}$ which implies  $L \subseteq  \sqrt{0}$.\\
$\Longleftarrow$ Suppose that $x \circ y \circ z \subseteq I$ for some nonunit elements $x,y,z \in R$ but $x \circ y \nsubseteq I$. It is clear that $\langle x \rangle  \circ \langle y \rangle \circ \langle z \rangle \subseteq I$. Since  $\langle x \rangle  \circ \langle y \rangle \nsubseteq I$, then we have  $\langle z \rangle\subseteq \sqrt{0}$ which implies  $z \in \sqrt{0}$. Thus, the hyperideal $I$ of $R$ is  strongly 1-absorbing primary.
\end{proof}

\begin{theorem} \label{57}
The hyperideal $\langle 0 \rangle$ is the only strongly 1-absorbing
primary hyperideal of $R$ if and only if $R$ is a hyperfield or a non local integral hyperdomain.
\end{theorem}
\begin{proof}
$\Longrightarrow$ Let $R$ is a local multiplicative hyperring with maximal hyperideal $M$. Since hyperideal $M$ is 
strongly 1-absorbing primary, then $M=\langle 0 \rangle$ which means $R$ is a hyperfield. Suppose that $R$ is a non local. Since hyperideal $\sqrt{0}$ is  strongly 1-absorbing primary, then we have $\sqrt{0}=\langle 0 \rangle $ is a prime hyperideal which means $R$ is a integral hyperdomain.\\
$\Longleftarrow$ Clearly, If $R$ is a hyperfield or a non local integral hyperdomain then hyperideal $\langle 0 \rangle $ is the only strongly 1-absorbing primary hyperideal of $R$.
\end{proof}
\begin{lem} \label{58}
Let the hyperideal $I$ of $R$  be  1-absorbing primary  and let $J$ be a proper hyperideal of $R$ such that $J \nsubseteq  I$. Then  $(I : J)$ is a primary hyperideal of $R$.
\end{lem}
\begin{proof}
Let for some $x,y \in R$, $x \circ y \subseteq (I:J)$ but $x \notin (I:J)$. Hence $y $ is a nonunit elements of $R$. The element $x$ is a nonunit element of $R$ as well. Since $x \notin (I:J)$, then for some $z \in J$, $x \circ z \nsubseteq I$. Since $I$ is a 1-absorbing primary hyperideal of $R$ and $x \circ y \circ z \subseteq I$, then we have $y \in \sqrt{I} \subseteq \sqrt{(I:J)}$ which means $(I : J)$ is a primary hyperideal of $R$.
\end{proof}
\begin{theorem} \label{59}
Let the hyperideal $I$ of $R$  be  1-absorbing primary  and let $J$ be a proper hyperideal of $R$ such that $J \nsubseteq  \sqrt{I}$. Then  hyperideal $(I : J)$ of $R$ is  strongly 1-absorbing  primary.
\end{theorem}
\begin{proof}
Let hyperideal $I$ of $R$  be  1-absorbing primary  and let $J$ be a proper hyperideal of $R$. Suppose that $\sqrt{I} \neq \sqrt{0}$. Hence $R$ is a local multiplicative hyperring with maximal hyperideal $\sqrt{I}$ and $J \subseteq \sqrt{I},$ which is a contradicion. Then we assume that $\sqrt{I}=\sqrt{0}$ is a prime hyperideal of $R$. Assume that $a \in (I:J)$. This means $a \circ J \subseteq I \subseteq \sqrt{I}$. Since hyperideal $I$ of $R$  be  1-absorbing primary and $J \nsubseteq \sqrt{I}$, then $a \in \sqrt{I}$. Thus, we have $I \subseteq (I:J) \subseteq \sqrt{I}$. Consequently, $\sqrt{(I:J)}=\sqrt{I}=\sqrt{0}$. Then by Lemma \ref{58} and Theorem \ref {51}, we conclude that hyperideal $(I : J)$ of $R$ is  strongly 1-absorbing  primary.
\end{proof}
We close this section by investigation the stability of strongly 1-absorbing primary hyperideals in various
ring-theoretic constructions.
\begin{theorem} \label{42} 
Let $R_1$ and $R_2$ be multiplicative hyperrings and $ \phi:R_1 \longrightarrow R_2$ be a good homomorphism. Then the following statements hold : 
\begin{itemize}
\item[{(1)}]~ If $\phi$ is a monomorphism and $I_2$ is a strongly 1-absorbing primary hyperideal of $R_2$, then $\phi^{-1} (I_2)$ is a strongly 1-absorbing primary hyperideal of $R_1$.
\item[{(2)}]~ If $\phi $ is an epimorphism and $I_1$ is a strongly   1-absorbing primary hyperideal of $R_1$ containing $Ker(\phi)$, then $\phi(I_1)$ is a strongly 1-absorbing primary hyperideal of $R_2$.
\end{itemize}
\end{theorem} 
\begin{proof} 
(1)  Let for some nonunit elements  $x,y,z \in R_1$, $x \circ  y \circ z \subseteq \phi^{-1}(I_2)$ such that $x \circ y \nsubseteq \phi^{-1}(I_2)$. This means $\phi(x) \circ \phi(y) \nsubseteq I_2$. Since $\phi(x) \circ \phi(y) \circ \phi(z) \subseteq I_2$ and hyperideal $I_2$ of $R_2$ is strongly 1-absorbing primary, then $\phi(z) \in \sqrt{0_{R_2}}$ which implies $z \in \sqrt{0_{R_1}}$, because $\phi$ is  a monomorphism. Thus  $\phi^{-1} (I_2)$ is a strongly 1-absorbing primary hyperideal of $R_1$.\\
(2) Let for some nonunit elements $x_2,y_2,z_2 \in R_2$,  $x_2 \circ y_2 \circ z_2 \subseteq \phi(I_1)$. Since $\phi$ is an epimorphism, there exist $x_1,y_1,z_1 \in R_1$ such that $\phi(x_1)=x_2, \phi(y_1)=y_2, \phi(z_1)=z_2$ and so $\phi(x_1 \circ y_1 \circ z_1)=x_2 \circ y_2 \circ z_2\subseteq \phi(I_1)$. Now take any $u \in x_1 \circ  y_1 \circ z_1 $. Then we get $\phi(u) \in \phi(x_1 \circ y_1 \circ z_1) \subseteq \phi(I_1)$ and so $\phi(u) = \phi(w)$ for some $w \in I_1$. This implies that $\phi(u-w) = 0 $, that is, $u-w \in Ker(\phi) \subseteq I_1$ and so $u \in I_1$. Since $I_1$ is a ${\bf C}$-hyperideal of $R_1$, then we conclude that $x_1 \circ y_1 \circ z_1\subseteq I_1$. Since $I_1$ is a strongly 1-absorbing primary hyperideal of $R_1$ then $x_1 \circ y_1 \subseteq I_1$ which implies $x_2 \circ y_2 =\phi(x_1 \circ y_1) \subseteq \phi(I_1)$ or $z_1 \in \sqrt{0_{R_1}}$ which implies $z_2=\phi(z_1) \in  \sqrt{0_{R_2}}$ .Thus, $\phi(I_1)$ is a strongly 1-absorbing primary hyperideal  of $R_2$. 
\end{proof}
\begin{corollary}
Let $I, J$ be two proper hyperideals of $R$ with $J \subseteq I$.
If  $I$ is a strongly 1-absorbing primary hyperideal of $R$ then  $I/J$ is a strongly  1-absorbing primary hyperideal of $R/I$.
\end{corollary}
\begin{corollary}
Let $T$ be a subhyperring of $R$. If hyperideal $I$ of $R$ is strongly 1-absorbing primary then hyperideal $T \cap I$ of $T$ is strongly 1-absorbing primary.
\end{corollary}
\section{weakly 1-absorbing primary hyperideal}
\begin{definition}
Let $I$ be a proper hyperideal of $R$. We call $I$
a weakly 1-absorbing primary hyperideal of $R$ if for nonunit elements $x,y,z \in R$, $0 \notin x \circ y \circ z \subseteq I$, then $x \circ y \subseteq I$ or $z \in \sqrt{I}$.
\end{definition}
It is clear that every  1-absorbing primary hyperideal of $R$ is a weakly 1-absorbing primary hyperideal. Moreover, every proper hyperideal of local multiplicative hyperring $R$ with maximal hyperideal $\sqrt{0}$ is  a weakly 1-absorbing primary hyperideal of $R$.
\begin{example}
Let $(\mathbb{Z},+,.)$ be the ring of integers. We define the hyperoperation $a \circ b=\{2a.b,4a.b\}$, for all $a,b \in \mathbb{Z}$. Then $(\mathbb{Z},+,\circ)$ is a multiplicative hyperring. The subset $<15>=\{15n \ \vert \ n \in \mathbb{Z}\}$ is a weakly 1-absorbing primary hyperideal of $\mathbb{Z}$.
\end{example}
\begin{theorem}   \label{61}
Let  $I$ be a weakly 1-absorbing primary ideal of $R$.
If hyperideal $\sqrt{I}$ of $R$  is  maximal, then  $I$ is  primary and then $I$ 
is 1-absorbing primary.
\end{theorem}
\begin{proof}
Assume that  $I$ is a weakly 1-absorbing primary ideal of $R$ such that hyperideal $\sqrt{I}$ of $R$  is  maximal. This means that $\sqrt{I}$ is a prime hyperideal of $R$. Hence $I$ is a primary hyperideal of $R$. Consequently, hyperideal $I$ of $R$ is 1-absorbing primary.
\end{proof}
\begin{theorem} \label{62}
Let $R$ be a reduced hyperring and let $I$ be nonzero hyperideal of $R$. If $I$ is  weakly 1-absorbing primary, then $\sqrt{I}$ is a prime hyperideal of $R$. 
\end{theorem}
\begin{proof}
Assume that  for some nonunit  elements $x,y \in R$, $0 \neq x \circ y \subseteq  \sqrt{I}$. This means that $(x \circ y)^n \subseteq I$ for  $n=2t (t \geq 1)$. Since $R$ is a reduced hyperring $(x \circ y)^n \neq 0$. we have $0 \neq x^t \circ x ^t \circ y ^n \subseteq I$. Since $I$ is a weakly 1-absorbing primary hyperideal of $R$, then we get $x^t \circ x^t =x^n \subseteq I$ or $y^n \subseteq \sqrt{I}$. Since $I \neq \{0\}$, then $\sqrt{I}$ is a prime hyperideal of $R$.
\end{proof}
\begin{theorem} \label{63}
Let $I$ be a nonzero hyperideal of a regular hyperring $R$. Then the followings are equivalent:
\begin{itemize}
\item[{(1)}]~ The hyperideal $I$ of $R$ is weakly 1-absorbing primary.
\item[{(2)}]~ The hyperideal $I$ of $R$ is  primary. 
\item[{(3)}]~ The hyperideal $I$ of $R$ is 1-absorbing primary. 
\end{itemize}
\end{theorem}
\begin{proof}
$(1) \Longrightarrow (2)$ Since $R$ is a regular hyperring, then $R$ is a reduced hyperring, by Theorem 8 in \cite{amer3}. Therefore, by Theorem \ref{62}, hyperideal $\sqrt{I}$ of $R$ is prime. Since every prime hyperideal in a regular hyperring is maximal, then hyperideal $\sqrt{I}$ of $R$ is maximal. Thus hyperideal $I$ of $R$ is primary by Theorem \ref{61}. The proof of other cases is clear.
\end{proof}
\begin{definition}
Let hyperideal $I$ of $R$ be  weakly 1-absorbing primary and  $x,y,z$ be nonunit elements of $R$. We call $(x, y, z)$ is a 1-triple-zero of $I$ if $\{0\} = x \circ y \circ z$ , $x \circ y \nsubseteq I$ and  $z \notin \sqrt{I}$.
\end{definition}
\begin{theorem}
Let $R$ be a strongly distributive multiplicative hyperring. Let hyperideal $I$ of $R$ be weakly 1-absorbing primary  and  $(x, y, z)$ be a 1-triple-zero of $I$. Then
\begin{itemize}
\item[{(1)}]~ $0 \in x \circ y \circ I.$ 
\item[{(2)}]~ If $x, y \notin (I:z)$, then $0 \in y \circ z \circ I$ and $0 \in x \circ z \circ I$.
\end{itemize}
\end{theorem}
\begin{proof}
(1) Let $0 \notin x \circ y \circ I$. Then for every nonunit element $u \in I$, $0 \notin x \circ y \circ u$. Therefore we have $0 \notin x \circ y \circ (z+u) \subseteq I$. Since $(x,y,z)$ is a 1-triple-zero and $I$ is a weakly
1-absorbing primary hyperideal of $R$, then $z+u \in \sqrt{I}$. Since $u \in I \subseteq \sqrt{I}$ then $z \in \sqrt{I}$ which is a contradiction. Hence,  $0 \in x \circ y \circ I$.\\
(2) Let $0 \notin y \circ z \circ I$. Then for every  nonunit element $v \in I$ such that $0 \notin y \circ z \circ v$.
Therefore we have $0 \notin (x+v) \circ y
 \circ z \subseteq I$. Since $y \circ z \nsubseteq I$, then $x+v $ is  a nonunit element of $R$. Since $x \circ y \nsubseteq I$ and $v \circ y \subseteq I$ and hyperideal $I$ of $R$ is weakly 1-absorbing primary, then $ (x + v) \circ y \nsubseteq I$ which implies $z \in \sqrt{I}$ which is a contradiction. By a similar argument, we can conclude that $0 \in x \circ z \circ I$.
 \end{proof}
 \begin{theorem}
 Let $\{I_i\}_{i \in \Delta}$
be a set  of weakly 1-absorbing primary
hyperideals of $R$ such that for every distinct $i,j \in \Delta$, $P = \sqrt{I_i}=\sqrt{I_j}$. Then $\bigcap_{i \in \Delta}I_i$
is a weakly 1-absorbing primary hyperideal of $R$.
\end{theorem} 
\begin{proof}
Assume that $I=\bigcap_{i \in \Delta}I_i$. Let for nonunit elements $x,y,z \in R$, $0 \neq x \circ y \circ z \subseteq I$ such that $x \circ y \nsubseteq I$. This means that there exists $t \in \Delta$, $0 \neq x \circ y \circ z \subseteq I_t$ but $x \circ y \nsubseteq I_t$. Since $I_t$ is a weakly 1-absorbing primary hyperideal of $R$, then $z \in \sqrt{I_t}=P=\sqrt{I}$.
\end{proof} 
\begin{theorem} 
Let $R_1$ and $R_2$ be  multiplicative hyperrings with identity such that are not hyperfields.  Let $I$ be a nonzero proper hyperideal of $R_1 \times R_2$. Then the followings
are equivalent.
\begin{itemize}
\item[{(1)}]~ $I$ is a weakly 1-absorbing primary hyperideal of $R_1 \times R_2$.
\item[{(2)}]~ $I=R_1 \times I_2$ for some
primary hyperideal $I_2$ of $R_2$ or $I=I_1 \times R_2$ for some primary hyperideal $I_1$ of $R_1$.
\item[{(3)}]~$I$ is a 1-absorbing primary hyperideal of $R_1 \times R_2$.
\item[{(4)}]~$I$ is a primary hyperideal of $R_1$.
\end{itemize}
\end{theorem}
\begin{proof}
$(1) \Longrightarrow (2)$ Let the hyperideal $I$ of $R_1 \times R_2$ is  weakly 1-absorbing primary. Suppose that  for some hyperideals $ I_1$ of $R_1$ and $ I_2$ of $R_2$, we have $I=I_1 \times I_2$. Take any  $0 \neq z \in I_1$. We get $0 \neq (z,0) \subseteq (e_{R_1},0) \circ (e_{R_1},0) \circ (z ,e_{R_2})$ which implies $(e_{R_1},0) \circ (e_{R_1},0) \circ (z ,e_{R_2}) \subseteq I_1 \times I_2$. Then $(e_{R_1},0) \circ (e_{R_1},0) \subseteq I_1 \times I_2$ which implies $e_{R_1} \in I_1$ or $(z,e_{R_2}) \in \sqrt{I_1 \times I_2}=\sqrt{I_1} \times \sqrt{I_2}$ which implies $e_{R_2} \in I_2$. This is a contradiction. Hence, we assume that for some proper hyperideal $I=I_1  \times R_2$. Let for some nonunit  $x_1,y_1 \in R_1$, $x_1 \star y_1 \subseteq I_1$. Let $x_2$ be a nonunit element of $R_2$. Since  $0 \neq (x_1 \star y_1, x_2) \subseteq (x_1,e_{R_2}) \circ (e_{R_1}, x_2) \circ (y_1,e_{R_2})$ then  $(x_1,e_{R_2}) \circ (e_{R_1}, x_2) \circ (y_1,e_{R_2}) \subseteq I_1 \times R_2$. This means $(x_1,e_{R_2}) \circ (e_{R_1},x_2) \subseteq I_1 \times R_2$ which implies $x_1 \in I_1$ or $(y_1, e_{R_2}) \in \sqrt{I_1 \times R_2}=\sqrt{I_1} \times R_2$ which implies $y_1 \in \sqrt{I_1}$.\\
$(2) \Longrightarrow (3)$ The proof  follows from Theorem \ref{21} (1).\\
$(3) \Longrightarrow (4)$ The proof  follows from Theorem \ref{25}.\\
$(4) \Longrightarrow (1)$ It is obvious.
\end{proof}
\begin{theorem} \label{phi}
Let $R_1$ and $R_2$ be multiplicative hyperrings and $ \phi:R_1 \longrightarrow R_2$ be a good homomorphism. Then the followings hold : 
\begin{itemize}
\item[{(1)}]~ If $\phi$ is a monomorphism and for
every nonunit element $x \in R_1$, 
$\phi(x)$ is a nonunit element of $R_2$ and  $I_2$ is a weakly 1-absorbing primary hyperideal of $R_2$, then $\phi^{-1} (I_2)$ is a weakly 1-absorbing primary hyperideal of $R_1$.
\item[{(2)}]~ If $\phi $ is an epimorphism and $I_1$ is a weakly   1-absorbing primary hyperideal of $R_1$ containing $Ker(\phi)$, then $\phi(I_1)$ is a weakly 1-absorbing primary hyperideal of $R_2$.
\end{itemize}
\end{theorem} 
\begin{proof} 
(1)  Let for some nonunit elements  $x,y,z \in R_1$, $0 \neq x \circ  y \circ z \subseteq \phi^{-1}(I_2)$. Since $\phi$ is a monomorphism then  $0 \neq \phi(x) \circ \phi(y) \circ \phi(z) \subseteq I_2$. Since hyperideal $I_2$ of $R_2$ is weakly 1-absorbing primary, then $\phi (x) \circ \phi (y) \subseteq I_2$ which implies $x \circ y \subseteq \phi^{-1}(I_2)$ or $\phi(z) \in \sqrt{I_2}$ which implies $z \in \sqrt{\phi^{-1}(I_2)}=\phi^{-1}(\sqrt{I_2})$. Thus  $\phi^{-1} (I_2)$ is a weakly 1-absorbing primary hyperideal of $R_1$.\\
(2) Let for some nonunit elements $x_2,y_2,z_2 \in R_2$,  $0 \neq x_2 \circ y_2 \circ z_2 \subseteq \phi(I_1)$. Since $\phi$ is an epimorphism, there exist $x_1,y_1,z_1 \in R_1$ such that $\phi(x_1)=x_2, \phi(y_1)=y_2, \phi(z_1)=z_2$ and so $\phi(x_1 \circ y_1 \circ z_1)=x_2 \circ y_2 \circ z_2\subseteq \phi(I_1)$. Now take any $u \in x_1 \circ  y_1 \circ z_1 $. Then we get $\phi(u) \in \phi(x_1 \circ y_1 \circ z_1) \subseteq \phi(I_1)$ and so $\phi(u) = \phi(w)$ for some $w \in I_1$. This implies that $\phi(u-w) = 0 $, that is, $u-w \in Ker(\phi) \subseteq I_1$ and so $u \in I_1$. Since $I_1$ is a ${\bf C}$-hyperideal of $R_1$, then we conclude that $0 \neq x_1 \circ y_1 \circ z_1\subseteq I_1$. Since $I_1$ is a weakly 1-absorbing primary hyperideal of $R_1$ then $x_1 \circ y_1 \subseteq I_1$ which implies $x_2 \circ y_2 =\phi(x_1 \circ y_1) \subseteq \phi(I_1)$ or $z_1 \in \sqrt{I_1}$ which implies $z_2=\phi(z_1) \in  \phi(\sqrt{I_1})=\sqrt{\phi(I_1)}$ .Thus, $\phi(I_1)$ is a weakly 1-absorbing primary hyperideal  of $R_2$. 
\end{proof}
\begin{corollary}
Let  $I$ and $J$ be proper hyperideals of $R$ such that $J \subseteq I$. If  $I$ is a weakly 1-absorbing
primary hyperideal of $R$, then $I/J$ is a weakly 1-absorbing primary ideal of $R/J$.
\end{corollary}
\begin{proof}
Define $\phi :R \longrightarrow R/I$ by $\phi(r)=r+J$. It is clear that $\phi$  is a good epimorphism. Since $Ker (\phi)=J \subseteq I$ and $I$ is a weakly 1-absorbing
primary hyperideal of $R$, then the claim follows from
Theorem \ref{phi} (2).

\end{proof}
\begin{definition}
Let the hyperideal $I$ of $R$  be  weakly 1-absorbing primary and for some proper hyperideals $J,H,K$ of $R$, $J \circ H \circ K \subseteq I$. If for every $j \in J$ , $h \in H$, $k \in K$, $(j,h,k)$ is not a 1-triple zero, then $I$ is called a free 1-triple zero with respect to $J \circ H \circ K$.
\end{definition}
\begin{lem} \label{JHK}
Let the hyperideal $I$ of $R$ be weakly 1-absorbing primary  and $K$ be a proper hyperideal of $R$ such that for some $x,y \in R$, $x \circ y \circ K \subseteq I$. If for all $k \in K$, $(x,y,k)$ is not a 1-triple zero of $I$ and $x \circ y \nsubseteq I$, then $K \subseteq \sqrt{I}$.
\end{lem}
\begin{proof}
Let $K \subseteq \sqrt{I}$. This means that there exists $k \in K$ but $k \notin \sqrt{I}$. Then we have $x \circ y \circ k \subseteq x \circ y \circ K \subseteq I$. Since $(x,y,k)$ is not a 1-triple zero and $x \circ y \nsubseteq I$ and $x \circ y \circ z =0$, then we get $k \in \sqrt{I}$ which is a contradiction. Hence, $K \subseteq \sqrt{I}$.
\end{proof}
\begin{theorem}
Let the hyperideal $I$ of $R$ be  weakly 1-absorbing primary and for some proper hyperideals $J,H,K$ of $R$, $0 \neq J \circ H \circ K \subseteq I$. If $I$ is free 1-triple zero with respect to $J \circ H \circ K$, then $J \circ H \subseteq I$ or $K \subseteq \sqrt{I}$.
\end{theorem}
\begin{proof}
Let for proper hyperideals $J,H,K$ of $R$, $0 \neq J \circ H \circ K \subseteq I$ such that $J \circ H \nsubseteq I$. This means that there exist $j \in J$ and $h \in H$ such that $j \circ h \nsubseteq I$. For all $k \in K$, $(j,h,k)$ is not a
1-triple zero of $I$, because  $I$ is a free 1-triple zero with respect to $J \circ H \circ K$. By Lemma \ref{JHK}, we conclude that $K \subseteq \sqrt{I}$
\end{proof}

\end{document}